\documentclass{amsart}
\usepackage{amsmath, amsthm, amssymb}
\usepackage{amssymb}
\usepackage{epsfig}
\usepackage{amsfonts}
\usepackage{amsmath}
\usepackage{epsfig}
\usepackage{subfigure}
\usepackage{tikz}
\usetikzlibrary{matrix}
\usetikzlibrary{arrows}
\usepackage{rotating}

\theoremstyle{plain}
\newtheorem{theorem}{Theorem}[section]
\newtheorem{proposition}[theorem]{Proposition}
\newtheorem{lemma}[theorem]{Lemma}

\theoremstyle{definition}

\theoremstyle{remark}

\def\P{\mathbb P}

\def\O{\mathcal O}

\def\H{\mathcal H}
\def\J{\mathcal J}

\def\Hom{{\rm Hom}}

\def\W{\mathcal W}
\def\Ext{{\rm Ext}}
\newcommand{\Coh}{\mathcal Coh}

\begin{document}

\title[Finite local systems in the Drinfeld-Laumon construction]{Finite local systems in the Drinfeld-Laumon construction}

\author{Galyna Dobrovolska}

\begin{abstract}

Let $E$ be a local system on a smooth projective curve of genus $g$ with monodromy given by a representation of the symmetric group corresponding to a Young diagram with rows of lengths $n_1,n_2,...$ where $n_1 > n_2 + (2g-2)$, $n_2 > n_3 + (2g-2)$, ..., $n_{k-1} > n_k + (2g-2)$, $n_k > n_{k+1} + n_{k+2} + ... + (2g-2)$. We show that the result of $k$ steps of the Drinfeld-Laumon construction applied to $E$ is the IC sheaf of the Harder-Narasimhan stratum with subquotients of rank $1$ and degrees $n_1$, $n_2$, ..., $n_k$, $n_{k+1}+n_{k+2}+...$ with coefficients in a local system with monodromy given by the Young diagram with rows $n_{k+1}$, $n_{k+2}$, $...$

\end{abstract}

\maketitle

\section{Introduction}

Let $C$ be a smooth projective curve of genus g and let $E$ be a local system on $C$.
Consider the natural symmetrization map $sym: C^n \to C^{(n)}$, where $C^{(n)}$ is the $n$-th symmetric power of the curve $C$. 
The symmetrization $(sym_* (E^{\boxtimes n}))^{S_n}$ of the local system $E$ restricted to  the locus of $n$-tuples of unordered distinct points $C_{dist}^{(n)}$
gives rise to a local system on the moduli stack $\Coh^n_{dist, 0}$ of coherent sheaves of rank $0$ and length $n$ supported at $n$ distinct points  via the natural map from $C_{dist}^{(n)}$ to $\Coh^n_{dist, 0}$. Extending this local system to the stack $\Coh^n_0$ of coherent sheaves of rank $0$ and length $n$ from the open substack $\Coh_{dist, 0}$ by the Goresky-MacPherson extension gives the Springer-Laumon sheaf $W_E$ on $\Coh_0$.
We are interested in the case of the trivial local system $E$ of rank $k$.

While proving the geometric Langlands conjecture, Gaitsgory proved that for an irreducible local system $E$ of rank $k$
on $C$ the result of the Drinfeld-Laumon construction (see Section 3.1) applied to $W_E$ after $k$ steps descends to the stack $Coh_{k+1}$ of coherent sheaves of rank $k+1$.
By analogy (due to D. Arinkin), we expect that the same should be true when we do this for the trivial local system $E$ of rank $k$.
Note that in this case the Springer-Laumon sheaf is constructed out of the symmetrization of the trivial local system of rank $k$.
By Schur-Weyl duality this symmerization is a direct sum of local systems corresponding to Young diagrams
with at most $k$ rows.
Hence we expect that after $k$ steps of the Drinfeld-Laumon construction are applied to the local system associated
to a Young diagram with at most $k$ rows the resulting sheaf will descend to $Coh_{k+1}$. We confirm such an expectation by calculating below the result of the Drinfeld-Laumon construction applied to the sheaf $W_{\rho}$ on $\Coh_0^{n}$, corresponding to a representation $\rho$ of the symmetric group $S_n$, which is a direct summand of the Springer-Laumon sheaf corresponding to the trivial local system of rank $n$. 


On $\Coh_0^{n}$ we consider the sheaf $W_{\rho}$ corresponding to a representation $\rho$ of the symmetric group $S_n$ which is a direct summand of the Springer-Laumon sheaf corresponding to the trivial local system of rank $n$. Consider further the open substack $_{\rm inj }\Coh_1^{\prime, n}$ of the stack $\Coh_1^{\prime, n}$ the points of which are sheaves $M_n$ of rank $1$ with a section $\O \to M_n$ such that the section is an injective map of sheaves. We have a natural smooth map $_{\rm inj }\Coh_1^{\prime, n} \to \Coh_0^{n}$ which sends the pair $\O \to M_n$ to the cokernel of the section. On $\Coh_1^{\prime, n}$ we consider the sheaf $\W_{\rho}$ which is the IC extension from $_{\rm inj }\Coh_1^{\prime, n}$ of the pullback of $W_{\rho}$ to $_{\rm inj }\Coh_1^{\prime, n}$ via the above smooth map. 

Consider a Harder-Narasimhan stratum $\widetilde S$ in the stack $\Coh_{k+1}^{n+k(k+1)(g-1)}$ of coherent sheaves of rank $k+1$ and degree $n+k(k+1)(g-1)$, the points of which are sheaves of rank $k+1$ and degree $n+k(k+1)(g-1)$ such that all the summands in their Harder-Narasimhan filtration have rank $1$ and fixed distinct degrees, and let the summand of the smallest degree in the Harder-Narasimhan filtration have degree $m+k(2g-2)$. Consider the Harder-Narasimhan stratum $S$ in the stack $\Coh_{k+1}^{\prime, n+k(k+1)(g-1)}$ (the moduli stack of coherent sheaves $M_{k+1}$ of rank $k+1$ and degree $n+k(k+1)(g-1)$ with a section $\Omega^k \to M_{k+1}$) which is the preimage of $\widetilde S$ under the map $\Coh_{k+1}^{\prime, n+k(k+1)(g-1)} \to \Coh_{k+1}^{n+k(k+1)(g-1)}$ which forgets the section. Let $\mu$ be a Young diagram with content $m$. Consider the substack $_{\rm inj}S$ of $S$ for which the composition of the section and the projection onto the Harder-Narasimhan summand of the smallest degree is injective. There is a smooth map from the substack $_{\rm inj}S$ to $\Coh_0^m$ which sends a sheaf with a section to the quotient of the smallest Harder-Narasimhan summand by the section. Let $\W_{S,\mu}$ be the IC extension from $_{\rm inj}S$ to $\Coh_{k+1}^{\prime, n+k(k+1)(g-1)}$ of the pull-back to $_{\rm inj}S$ of the sheaf $W_{\mu}$ on $\Coh_0^m$

For a Young diagram $\rho$ with content $n$ and rows of lengths $n_1 > n_2 + (2g-2)$, $n_2 > n_3 + (2g-2)$, ..., $n_{k-1} > n_k + (2g-2)$, $n_k > n_{k+1} + n_{k+2} + ... + (2g-2)$. Let the Young diagram $\mu$ be obtained from $\rho$ by deleting from it the $k$ longest rows. Let $S$ be the Harder-Narasimhan stratum in $\Coh_{k+1}^{\prime,n+k(k+1)(g-1)}$ consisting of $\Omega^k \to M_{k+1}$ such that $M_{k+1}$ has the Harder-Narasimhan filtration with summands of rank $1$ and degrees $n_1, n_2 + (2g-2), n_3 + 2(2g-2), ... , n_k + (k-1)(2g-2), (n_{k+1} + n_{k+2} + ...) + k(2g-2)$. We prove

\begin{theorem}[Main Theorem] 
The result of the $k$-th step of the Drinfeld-Laumon construction applied to $\W_{\rho}$ is the sheaf $\W_{S,\mu}$ on $\Coh_{k+1}^{\prime,n+k(k+1)(g-1)}$.
\end{theorem}

\subsection*{Main Theorem of \cite{Do} as a Special Case of the Main Theorem of this paper}

In the case of $k=1$ the main result is equivalent to the main result of \cite{Do}. Note that the sheaf $W_{\rho}$ restricts to a sheaf of the form ${\bf L}_{\rho}^C$ on the fibers of the map $\Coh_1^{\prime} \to \Coh_1$ which are of the form $H^0(M)$ for $M \in \Coh_1$.  Note also that the cones over the secant varieties in $H^0(M)^*$ are the intersections of the fibers of the map $^{0}\Coh_2^{\prime} \to Coh_1$ (the map from the stack of rank two coherent sheaves with an injective section to the stack of rank one sheaves by taking the cokernel of the section) with the closures of the preimages under the map $^{0}{\rm Bun}_2^{\prime} \to {\rm Bun}_2$ of the Harder-Narasimhan strata in ${\rm Bun}_2$, as shown below. 

For a curve $C$ and a line bundle $M$ on $C$, we have the map $C \to \P( H^0(M)^*)= \P Ext^1 (M,K)$ 
given by the line bundle $M$. It can be described as follows. For a point $x \in C$ we consider a map 
$f_x: M \to \O_x$ with kernel $M(-x)$ (this map is unique up to scaling). We pull back the 
extension $0 \to K \to K(x) \to \O_x \to 0$ under the map $f_x$. This is the element $h_x$ in $\P Ext^1(M,K)$
corresponding to $x$. Notice that $h_x|_{M(-x)}=0$. For points $x_1,...,x_s$ on the curve $C$ we 
also have $\Sigma a_i h_{x_i}|_{M(-x_1-...-x_s)}=0$; this translates into the statement that a point 
representing a bundle in $\P Ext^1(M,K)$ lies on the $s$-th secant variety if and only if it has a subbundle
of degree $n-s$, which for rank $2$ bundles is the description of a Harder-Narasimhan stratum.



\subsection*{Organization of the paper}

In Section 2 we introduce the Drinfeld-Laumon construction and Hecke functors.
In Section 3 we study how Hecke functors act on components of the Springer-Laumon sheaf and on IC sheaves of Harder-Narasimhan strata. 
In Section 4 we show compatibility of the Drinfeld-Laumon construction with Hecke functors.
In Section 5 we show Whittaker nondegeneracy of IC sheaves of Harder-Narasimhan strata filtered by line bundles.
In Section 6 we calculate the support of the Drinfeld-Laumon construction applied to a component of the Springer-Laumon sheaf.
Finally, in Section 7 we prove the Main Theorem by induction on the Young diagram.  

\subsection*{Acknowledgements}
The author is very grateful to Roman Bezrukavnikov for suggesting this problem and for many useful discussions and ideas. The author was supported by an NSF Postdoctoral Fellowship. 

\section{Background}

For a smooth projective curve $C$ we can consider the stack $\Coh_k(C)$ of coherent sheaves of rank $k$ on $C$ and the stack $\Coh_k^m(C)$ of coherent sheaves of rank $k$ and degree $m$ on $C$. Let $\Omega$ be the canonical line bundle of $C$.
We can also consider the stack $\Coh_k^{\prime}(C)$ of coherent sheaves on $C$ of rank $k$ with a section $\Omega^{k-1} \to M_k$ and the stack $\Coh_k^{\prime, m}(C)$ of coherent sheaves on $X$ of rank $k$ and 
degree $m$ with a section $\Omega^{k-1} \to M_k$. Similarly we define $^{0}\Coh_k^{\prime}(C)$ (resp. $^{0}\Coh_k^{\prime,m}(C)$) to be the stack of coherent sheaves on $C$ of rank $k$ (resp. of rank $k$ and degree $m$) with a section $\Omega^{k-1} \hookrightarrow M_k$ which is an injective map of sheaves. Fixing $C$, sometimes we will abbreviate $\Coh_k = \Coh_k(C)$, $\Coh_k^{m} =\Coh_k^m(C)$, $\Coh_k^{\prime} = \Coh_k^{\prime}(C)$, $\Coh_k^{\prime,m} = \Coh_k^{\prime, m}(C)$, $^{0}\Coh_k^{\prime} = ^{0}\Coh_k^{\prime}(C)$, and $^{0}\Coh_k^{\prime,m} = ^{0}\Coh_k^{\prime,m}(C)$. 

\subsection*{The Drinfeld-Laumon Construction}

We have a map $\pi_k: \Coh_k^{\prime,m} \to \Coh_k^m$ which sends $\Omega^{k-1} \to M_k$ to $M_k$. We also have a map $\pi^{\vee}_k:\ ^{0}\Coh_{k+1}^{\prime,m+(2g-2)k} \to \Coh_k^m$ which sends $\Omega^k \hookrightarrow M_{k+1}$ to $M_{k+1}/\Omega^k$. For these maps to be dual vector bundles we need to take an open substack of $\Coh_k^m$ such that the fibers of the maps $\pi_k$ and $\pi^{\vee}_k$ have constant dimension over it. Therefore we introduce the open substack  $_{(2g-2)k}\Coh_k^m$ of $\Coh_k^m$ consisting of $\Omega^{k-1} \to M_k$ where $M_k$ satisfies $\Hom(M_k, \Omega^k) =0$ which has the above property. We denote by $_{(2g-2)k}^{0}\Coh_k^{\prime,m}$ the stack of $\Omega^{k-1} \hookrightarrow M_k$ such that $\Hom(M_k, \Omega^k) =0$ and by $_{(2g-2)k}\Coh_k^{\prime,m}$ the stack of $\Omega^{k-1} \to M_k$ such that $\Hom(M_k, \Omega^k) =0$. 


Let $\Phi_k: D^b_c(_{(2g-2)k}\Coh_k^{\prime,m}) \to D^b_c((\pi_k^{\vee})^{-1}(_{(2g-2)k}\Coh_k^m))$ be the Fourier-Deligne transform for the dual vector bundles $\pi_k$ and $\pi^{\vee}_k$. Suppose that we are given a sheaf ${\mathcal F}_1$ on $\Coh_1^{\prime}$. Define inductively ${\mathcal F}_{k+1} = (i_{k+1})_{!*}\Phi_k({\mathcal F}_k)$. We call ${\mathcal F}_{k+1}$ the result of applying $k$ steps of the Drinfeld-Laumon construction (which we sometimes abbreviate as the DL construction) to ${\mathcal F}_1$. For more information about the Drinfeld-Laumon construction we refer the reader to \cite{La}, \cite{La1}, or \cite{G}.

\subsection*{Hecke Functors}

We fix a point $x$ on $C$ and consider the Hecke correspondence $\H=\H^m_k$ the points of which are $M_k \subset M_k^{\prime}$ such that $M_k$ and $M_k^{\prime}$ have rank $k$ and ${\rm deg}(M_k)=m$, ${\rm deg}(M_k^{\prime})=m+1$. We have the maps $q: \H \to \Coh_k^{m+1}$ and $p: \H \to \Coh_k^m$ which send $M_k \subset M_k^{\prime}$ to $M_k^{\prime}$ and to $M_k$ respectively. The Hecke functor $T_k$ is defined as $p_! q^*$.

We also consider the Hecke correspondence $\J=\J^m_k$ for sheaves with sections the points of which are $\Omega^{k-1} \to M_k \subset M_k^{\prime}$ such that $M_k$ and $M_k^{\prime}$ have rank $k$ and ${\rm deg}(M_k)=m$, ${\rm deg}(M_k^{\prime})=m+1$. We have the maps $q: \H \to \Coh_k^{m+1}$ and $p: \H \to \Coh_k^m$ which send $\Omega^{k-1} \to M_k \subset M_k^{\prime}$ to $\Omega^{k-1} \to M_k^{\prime}$ and to $\Omega^{k-1} \to M_k$ respectively. The Hecke functor $T_k$ for sheaves with a section (which is the one we will use here) is defined as $p_! q^*$.

Note that both kinds of Hecke functors as well as their properties are described in detail in \cite{La}.

\section{Action of the Hecke functors}

We will prove the theorem by induction on the Young diagram with the help of the Hecke functors. For this we
need to study the Hecke functors and apply them to our sheaves as described below.

\subsection{Hecke Functor on $\Coh_1^{\prime}$} \label{hecke_1}

In this subsection we calculate the Hecke functor $T_1$ applied to the sheaf $\W_{\rho}$ on $\Coh_1^{\prime, n}$ (see the proposition below). To do this, we first calculate the Hecke functor $T_0$ applied to the sheaf $W_{\rho}$ on $\Coh_0^{n}$ in a way which is similar to a calculation in \cite{La}; then we use a lemma in \cite{La} to relate the Hecke functor on $\Coh_0^{n}$ with the Hecke functor on $\Coh_1^{\prime, n}$.

\begin{proposition} The functor $T_1$ sends the sheaf $\W_{\rho}$ to the sheaf $\W_{\rm Res (\rho)}$ where $Res (\rho)$ is the restriction of the representation $\rho$ of $S_n$ to $S_{n-1}$.
\end{proposition}

\begin{proof}
To calculate $T_0$ on $Coh_0^{n}$ we first calculate it on the open substack where the support of the torsion sheaf in $\Coh_0^{n}$ consists of $n$ distinct points. In that case the application of the Hecke functor has the effect of fixing one point in the support and letting the other points be arbitrary, so the Hecke functor for our sheaf reduces to the restriction from $S_n$ to $S_{n-1}$. That is we obtain that the result of 
the Hecke functor restricted to the open part with distinct points in the support is equal to the restriction of $L_{Res (\rho)}$ to the same open part.

To complete the calculation of the Hecke functor on $\Coh_0^{n}$ it remains to show that the result of
this Hecke functor is a perverse sheaf which is the intermediate extension of its restriction from the above open part. This is done as in the proof of Theorem 4.1 in \cite{La} using the
 relationship with the Springer resolution and the smallness of certain related maps. (Note that Laumon's proof works if we replace his local system $L$ by the trivial local system
thus obtaining our sheaf as a direct summand in the direct image of the constant sheaf by Laumon's map $\pi$; hence Laumon's proof works in our case to show that the result of
the Hecke functor is a perverse sheaf which is the intermediate extension of its restriction to the above open part).

To pass from $Coh_0$ to $Coh_1^{\prime}$ we use Lemma 4.2(i) in \cite{La} which says that the Hecke functor sends a pullback from $\Coh_0$ to $\Coh_1^{\prime}$ to a pullback from $\Coh_0$ to $\Coh_1^{\prime}$, which finishes the proof.
\end{proof}

\subsection{Hecke Functor on $\Coh_k^{\prime}$} \label{hecke_k}

Let $S$ be a Harder-Narasimhan stratum in $\Coh_k^{\prime,n}$. We will assume that the difference of the degrees of the two subquotients
of the greatest degrees in the Harder-Narasimhan filtration of a sheaf in $S$ is at least $2$ (note that only such sheaves will appear
in the statement of the Main Theorem). Let $S_{\rm new}$ be the Harder-Narasimhan
stratum in $\Coh_{k-1}^{\prime,n}$ such that the degrees of all but the largest one of its Harder-Narasimhan summands are the 
same as the degrees of the corresponding Harder-Narasimhan summands of $S$ and the degree of the largest Harder-Narasimhan 
summand of $S_{\rm new}$ is one less than the degree of the largest Harder-Narasimhan summand of $S$.

Recall the sheaf $\W_{S,\mu}$ on $\Coh_k^{\prime,n}$ which is described in the Introduction to this Chapter. In this subsection 
we will apply the Hecke functor to a sheaf of this form. Recall the substack $_{\rm inj}S$ of $S$ for which the composition of the section and the projection onto the Harder-Narasimhan summand of the smallest degree is injective. We will prove the following

\begin{proposition} \label{hecke_k_prop}
The result of the Hecke functor applied to $\W_{S,\mu}$ is supported on the closure of the HN stratum $S_{\rm new}$ and its restriction to the open substack $_{\rm inj}S_{\rm new}$ of $S_{\rm new}$ coincides with the restriction to $_{\rm inj}S_{\rm new}$ of the sheaf $\W_{S_{new},\mu}$.
\end{proposition}

Before proving the proposition, we point out the following lemma which will be used in the proof of the proposition (the proof of the lemma is
straightforward).

\begin{lemma} Let $M_k \subset M_k^{\prime}$ be an inclusion of two sheaves of rank $k$ with the difference of degrees equal to $1$. The Harder-Narasimhan filtration of $M_k^{\prime}$ has all summands of the same degrees as the corresponding Harder-Narasimhan summands of $M_k$ except
one of them which is an extension of $\O_x$ by the corresponding Harder-Narasimhan summand of $M_k$. 
\end{lemma}

Now we prove the proposition.

\begin{proof}[Proof of Proposition \ref{hecke_k_prop}]

The statement about the support of the result of the Hecke functor follows from the above lemma. We will now prove the statement about the restriction to $_{\rm inj}S_{\rm new}$. Recall the maps $q: \J \to \Coh_k^{\prime, n}$ and $p: \J \to \Coh_k^{\prime, n-1}$ and the maps $t_S:\ _{\rm inj}S \to \Coh_0$ and $t_{S_{\rm new}}: \ _{\rm inj}S_{\rm new} \to \Coh_0$. 

Note that by the lemma above $q^{-1}(\overline {_{\rm inj}S}) \cap p^{-1}(_{\rm inj}S_{\rm new})$ coincides with $q^{-1}(_{\rm inj}S) \cap p^{-1}(_{\rm inj}S_{\rm new})$ where $\overline {_{\rm inj}S}$ denotes the closure of $_{\rm inj}S$. We denote these equal stacks by $\J_S$. Let $q_S, p_S$ be the restrictions of the 
maps $p,q$ to $\J_S$. Note that $t_S \circ q_S$ coincides with $t_S \circ p_S$ (so in particular the sheaf $(t_S \circ q_S)^*\W_{S,\mu}$ is constant on the fibers of the map $p_S$). This fact combined with the fact that the fiber over $M_k$ of the map $p_S$ is the affine space $\Ext^1(\O_x, L)$ (where $L$ is the HN summand of the greatest degree in $M_k$) shows 
that the Hecke functor applied to $\W_{S,\mu}$ gives $\W_{S_{\rm new},\mu}$ when restricted to $_{\rm inj}S_{\rm new}$.

\end{proof}

\section{Commutation of the DL construction with the Hecke functors up to degenerate summands} \label{hecke_commute}

The proof is modelled on the proof in \cite{La}, page 17. The proof that the Fourier transforms in the DL construction commute with the Hecke functors is exactly as in \cite{La}. The proof in \cite{La} that the restrictions under the maps $i_k$ commute with the Hecke functors goes through after a slight modification due to a slight difference in the definition of the DL construction for a curve of genus $0$ and a curve of positive genus. In this subsection we will show that these facts automatically imply that the intermediate extensions under the maps $i_k$ in the DL construction commute with the Hecke functors up to degenerate summands, in the sense of the proposition below. 


Let $L_k: D^b(\Coh^{\prime}_1) \to D^b(\Coh^{\prime}_{k+1})$ denote the first k steps of the Laumon construction with intermediate extensions.
Let $T_1: D^b(\Coh_1^{\prime,m+1}) \to D^b(\Coh_1^{\prime,m})$ be the Hecke functor.
Let $T_{k+1}: D^b(\Coh_{k+1}^{\prime,d+1}) \to D^b(\Coh_{k+1}^{\prime,d})$ be the Hecke functor (where $d = m - k(k+1)$).
Let $R_k$ be the result of $k$ steps of the inverse Laumon construction defined for $\mathcal F_{k+1}$ on $\Coh^{\prime}_{k+1}$ inductively by $R_0(F_{k+1}) = F_{k+1}$ and $R_{s+1} (\mathcal F) = \Phi_{k-s}^{-1}( i_{k-s+1}^* R_s(\mathcal F_{k+1}))$.

\begin{proposition} 
Given an object $A$ in $D^b(Coh'_1)$, each irreducible subquotient $C$
in each perverse cohomology of $T_{k+1}(L_k(A))$ is either $L_k(B)$ where $B$ is 
an irreducible subquotient in one of the perverse cohomology sheaves of the 
object $T_1(A)$ or is degenerate (i.e. is killed by the inverse Laumon construction
$R_k$).
\end{proposition}

\begin{proof}
We use the fact that both the restriction to an open substack and 
the Fourier transform are t-exact with respect to the perverse t-structure.
In this way we show that the inverse Laumon construction either kills $C$ or
maps it to one of the irreducible subquotients of a perverse cohomology sheaf 
of $R_k(T_{k+1}(L_k(A)))$. But the inverse Laumon construction commutes with the Hecke functors so $R_k(T_{k+1}(L_k(A))) = T_1(R_k L_k (A))$ by the remark above about the commutation of the Hecke functors with the Fourier transforms and the restrictions under the maps $i_k$ in the DL construction. Then we use the fact that the forward Laumon construction with intermediate extensions followed by the inverse Laumon construction gives the identity so $T_1(R_k L_k (A))=T_1(A)$ so we are done.
\end{proof}

\section{Nondegeneracy of sheaves supported on a HN stratum with subquotients of rank 1}  \label{nondeg}

We show that the IC sheaf of the preimage of a HN stratum with subquotients of rank $1$ which occurs in the Main Theorem is Whittaker non-degenerate. For this it is enough to show that the singular support of this sheaf intersects non-trivially the common open substack of the cotangent stacks of the stacks of sheaves with sections which is described on page 44 of \cite{La}. In turn for this it is enough to exhibit a coherent sheaf with a section $\Omega^{n-1} \hookrightarrow M_n$ in the open substack $_{\rm inj} S$ of the Harder-Narasimhan stratum $S$ described in the statement of the Main Theorem and a Higgs field $\theta: M_n \to M_n \otimes \Omega$ in the conormal bundle to the HN stratum such that their combination is in the common open substack of sheaves with sections decribed on p. 44 of \cite{La}. Below we reproduce use the diagram on page 44 of \cite{La} describing the common open substack of the stacks of sheaves with sections. Note that in this diagram the rank of $M_i$ is $i$, the horizontal sequences are short exact, and the diagram commutes.    


\begin{tikzpicture}
  \matrix (m) [matrix of math nodes,row sep=1em,column sep=8em,minimum width=2em, ampersand replacement=\&]
  {
     \Omega^{n-1} \& \ \ \  \  \  M_n \  \  \ \  \  \&  \ \  \  \ M_{n-1} \\
     \Omega^{n-1} \& M_{n-1} \otimes \Omega \& M_{n-2} \otimes \Omega \\
\vdots \& \vdots \& \vdots \\
\Omega^{n-1} \& M_1 \otimes \Omega^{n-1} \& M_0 \otimes \Omega^{n-1} \\};
  \path[-stealth]

    (m-1-1) edge node [right] {$=$} (m-2-1)
    (m-2-1) edge node [right] {$=$} (m-3-1)
    (m-3-1) edge node [right] {$=$} (m-4-1)

    (m-1-2) edge node [right] {\scriptsize{$\alpha_n$}} (m-2-2)
    (m-2-2) edge node [right] {\scriptsize{$\alpha_{n-1} \otimes \Omega$}} (m-3-2)
    (m-3-2) edge node [right] {\scriptsize{$\alpha_2 \otimes \Omega^{n-2}$}} (m-4-2)

    (m-1-3) edge node [right] {\scriptsize{$\alpha_{n-1}$}} (m-2-3)
    (m-2-3) edge node [right] {\scriptsize{$\alpha_{n-2} \otimes \Omega$}} (m-3-3)
    (m-3-3) edge node [right] {\scriptsize{$\alpha_1 \otimes \Omega^{n-2}$}} (m-4-3)
          
    (m-1-2) edge node [above] {\scriptsize{$\beta_n$}} (m-1-3)
    (m-2-2) edge node [above] {\scriptsize{$\beta_{n-1} \otimes \Omega$}} (m-2-3)
    (m-4-2) edge node [above] {\scriptsize{$\beta_1 \otimes \Omega^{n-1}$}} (m-4-3);


    \path[right hook->] 

     (m-1-1) edge node [above] {\scriptsize{$s$}} (m-1-2)

     (m-2-1) edge node [above] {\scriptsize{$s \otimes \Omega$}} (m-2-2)

     (m-4-1) edge node [above] {\scriptsize{$s \otimes \Omega^{n-1}$}} (m-4-2);

\end{tikzpicture}

\begin{lemma} If a section $s: \Omega^{n-1} \hookrightarrow M_n$ remains injective when composed with the quotient map to the smallest HN summand of $M_n$ then we can construct a diagram as above with this $s$.
\end{lemma}

\begin{proof}
Let $\theta: M_n \to M_n \otimes \Omega$ be a nilpotent of order $n$ which has the kernel equal to the greatest HN summand and maps each of the other HN summands to the one which is greater than it tensored by $\Omega$ (note that our HN strata have the difference at least $2$ between the degrees of the HN summands because when we build the HN stratum from a Young diagram we subtract from each row of the Young diagram $2$ less than from the next smaller row). 

We show that the Higgs field $\theta$ is in the conormal bundle to the HN stratum $S$ in the Main Theorem. Since the HN stratum consisting of sheaves of the form $\oplus_i A_i$ is the image of $\prod_i {\rm Pic}^{{\rm deg} A_i}$, the tangent space to this HN stratum is $\oplus \Ext^1(A_i,A_i)$. Since the Higgs field $\theta$ maps $A_i$ to $A_{i-1} \otimes \Omega$, in other words lies in $\oplus_i \Hom(A_i,A_{i-1} \otimes \Omega)$, which pairs to $0$ with $\oplus_i \Ext^1(A_i,A_i)$ by Serre duality, we get that $\theta$ is in the conormal bundle of the HN stratum $S$.

Given a section $s$ we can construct the diagram below by induction on $r$ defining the map $\beta_{n-r}$ as the quotient map of $M_{n-r}$ by $\alpha_{n-r+1} \otimes \Omega^{-1} ({\rm Ker}(\beta_{n-r+1} \otimes \Omega^{-1}))$, and the map $\alpha_{n-r}$ as the map induced by $\alpha_{n-r}$ on the corresponding quotients. This inductive construction is possible provided that the maps $\alpha^j = (\alpha_{j+1} \otimes \Omega^{n-j})(\alpha_{j+2} \otimes \Omega^{n-j-1}) \dots (\alpha_n \otimes \Omega)(s \otimes \Omega)$ are injective for every $j$. We will prove this by decreasing induction on $j$, the base being that $s \otimes \Omega$ is injective. Suppose that we proved the injectivity of $\alpha^{j^{\prime}}$ for all $j^{\prime}>j$. We want to prove that $\alpha^j$ is injective.

Assume that on the contrary $\alpha^j$ is not injective. We will formally manipulate the diagram as a diagram of modules over a ring using the Freyd-Mitchell embedding theorem. Therefore there is a nonzero $a$ such that $\alpha^j(a)=0$. By the inductive assumption we have $\alpha^{j+1}(a) \neq 0$ and $(\alpha_{j+1} \otimes \Omega^{n-j}) \alpha^{j+1} = 0$, hence 
\begin{equation}  \label{base}
\alpha^{j+1}(a) \in {\rm Ker}(\alpha^{j+1} \otimes \Omega^{n-j}).
\end{equation}

Let $\theta^i = (\theta \otimes \Omega^{n-j-1})(\theta \otimes \Omega^{n-j-2}) \dots (\theta \otimes \Omega^{n-i-1})$ and let $\beta^i = (\beta_{n+j-i} \otimes \Omega^{n-j})(\beta_{n+j-i+1} \otimes \Omega^{n-j}) \dots (\beta_n \otimes \Omega^{n-j})$. We are going to prove the following claim by decreasing induction on $i$:

Claim. There are elements $a_2, a_3, ... , a_{n - i - 1}$ such that the element
$$\beta^i \circ \theta^{n-2} \circ (s \otimes \Omega)(a) + \beta^i \circ \theta^{n-3} \circ (s \otimes \Omega^2)(a_2) + \dots + \beta^i \circ \theta^i \circ (s \otimes \Omega^{n-i-1})(a_{n-i-1})$$
is in ${\rm Ker}(\alpha_{n+j-i-1} \otimes \Omega^{n-j})$.
 
We are proving this claim by decreasing induction on $i$ and the base of induction for $i = n-2$ is provided by the statement (\ref{base}) above since $\alpha^{j+1}(a) = \beta^{n-2} \circ \theta^{n-2} \circ (s \otimes \Omega)(a)$ and according to (\ref{base}) this element is in ${\rm Ker}(\alpha_{j+1} \otimes \Omega^{n-j})$.

\newpage

\begin{sidewaysfigure}

\hspace*{-1cm}\vspace*{-10cm}\begin{tikzpicture}

\tikzstyle{every node}=[font=\scriptsize]

  \matrix (m) [matrix of math nodes,row sep=3.5em,column sep=1.5em,minimum width=0.5em, ampersand replacement=\&]
  {
     \ \& \  \&  \  \& \ \& \ \& \ \&  \ \& \Omega^{n-1} \& M_n \\
     \ \& \  \&  \  \& \ \& \ \& \ \&  \Omega^n \& M_n \otimes \Omega \& M_{n-1} \otimes \Omega  \\
     \ \& \  \&  \  \& \ \& \ \& \vdots \&  \vdots \& \vdots \& \vdots  \\
     \ \& \  \&  \  \& \ \& \Omega^{2n-i-2} \& M_n \otimes \Omega^{n-i-1} \& \dots \& M_{i+2} \otimes \Omega^{n-i-1} \& M_{i+1} \otimes \Omega^{n-i-1} \\
     \ \& \  \&  \  \& \Omega^{2n-i-1} \& M_n \otimes \Omega^{n-i} \& M_{n-1} \otimes \Omega^{n-i} \& \dots \& M_{i+1} \otimes \Omega^{n-i} \& M_i \otimes \Omega^{n-i} \\
     \ \& \ \&  \vdots \& \vdots \& \vdots \& \vdots \& \vdots \& \vdots \& \vdots \\
     \ \& \Omega^{2n-j-1} \& M_n \otimes \Omega^{n-j}  \& \dots \& M_{n+j-i} \otimes \Omega^{n-j} \& M_{n+j-i-1} \otimes \Omega^{n-j} \& \dots \& M_{j+1} \otimes \Omega^{n-j} \& M_j \otimes \Omega^{n-j} \\
     \Omega^{2n-j} \& M_n \otimes \Omega^{n-j+1} \& M_{n-1} \otimes \Omega^{n-j+1} \&  \dots \& M_{n+j-i+1} \otimes \Omega^{n+j-i} \& M_{n+j-i} \otimes \Omega^{n-j+1} \& \dots \& M_j \otimes \Omega^{n-j+1} \& M_{j-1} \otimes \Omega^{n-j+1} \\};
  \path[-stealth]

     (m-2-8) edge node [below] {\tiny{$\beta_n \otimes \Omega$}} (m-2-9)

     (m-4-6) edge node [below] {\tiny{$\beta_n \otimes \Omega^{n-i-1}$}} (m-4-7)
     (m-4-7) edge  (m-4-8)
     (m-4-8) edge node [below] {\tiny{$\beta_{i+2} \otimes \Omega^{n-i-1}$}} (m-4-9)

     (m-5-5) edge node [below] {\tiny{$\beta_n \otimes \Omega^{n-i}$}} (m-5-6)
     (m-5-6) edge  (m-5-7)
     (m-5-7) edge  (m-5-8)
     (m-5-8) edge node [below] {\tiny{$\beta_{i+1} \otimes \Omega^{n-i}$}} (m-5-9)

     (m-7-3) edge  (m-7-4)
     (m-7-4) edge  (m-7-5)
     (m-8-2) edge  (m-8-3)
     (m-8-3) edge  (m-8-4)
     (m-8-4) edge  (m-8-5)
     (m-7-3) edge node [right] {\tiny{$\alpha_n \otimes \Omega^{n-j}$}} (m-8-3)
     (m-7-6) edge  (m-7-7)
     (m-7-7) edge  (m-7-8)
     (m-7-8) edge node [below] {\tiny{$\beta_{j+1} \otimes \Omega^{n-j}$}}  (m-7-9)
     (m-8-6) edge  (m-8-7)
     (m-8-7) edge  (m-8-8)
     (m-8-8) edge node [below] {\tiny{$\beta_j \otimes \Omega^{n-j+1}$}} (m-8-9)
     (m-7-8) edge node [right] {\tiny{$\alpha_{j+1} \otimes \Omega^{n-j}$}} (m-8-8)
     (m-7-9) edge  (m-8-9)

     (m-7-5) edge node [right] {\tiny{$\alpha_{n-i+1} \otimes \Omega^{n-j}$}} (m-8-5)
     (m-7-6) edge node [right] {\tiny{$\alpha_{n-i} \otimes \Omega^{n-j}$}} (m-8-6)
     (m-7-5) edge node [below] {\tiny{$\beta_{n+j-i} \otimes \Omega^{n-j}$}} (m-7-6)
     (m-8-5) edge node [below] {\tiny{$\beta_{n+j-i+1} \otimes \Omega^{n-j+1}$}} (m-8-6)



     (m-1-9) edge node [right] {\tiny{$\alpha_n$}} (m-2-9)
     (m-4-6) edge node [right] {\tiny{$\alpha_n \otimes \Omega^{n-i-1}$}} (m-5-6)
     (m-4-8) edge node [right] {\tiny{$\alpha_{i+2} \otimes \Omega^{n-i-1}$}} (m-5-8)
     (m-4-9) edge (m-5-9)

     (m-1-9) edge node [left] {\tiny{$\theta$}} (m-2-8)
     (m-4-6) edge node [left] {\tiny{$\theta \otimes \Omega^{n-i-1}$}} (m-5-5)
     (m-7-3) edge node [left] {\tiny{$\theta \otimes \Omega^{n-j}$}} (m-8-2);

    \path[right hook->] 

     (m-1-8) edge node [below] {\tiny{$s$}} (m-1-9)
     (m-2-7) edge node [below] {\tiny{$s \otimes \Omega$}} (m-2-8)
     (m-4-5) edge node [below] {\tiny{$s \otimes \Omega^{n-i-1}$}} (m-4-6)
     (m-5-4) edge node [below] {\tiny{$s \otimes \Omega^{n-i}$}} (m-5-5)
     (m-7-2) edge node [below] {\tiny{$s \otimes \Omega^{n-j}$}} (m-7-3)
     (m-8-1) edge node [below] {\tiny{$s \otimes \Omega^{n-j+1}$}} (m-8-2);

\end{tikzpicture}

\end{sidewaysfigure}

\newpage

Now we assume that the Claim holds for $i$ and prove it for $i-1$. The Claim for $i$ tells us that 
$(\beta_{n+j-i} \otimes \Omega^{n-j})(\beta^{i-1} \circ \theta^{n-2} \circ (s \otimes \Omega)(a) + \beta^{i-1} \circ \theta^{n-3} \circ (s \otimes \Omega^2)(a_2) + \dots + \beta^{i-1} \circ \theta^i \circ (s \otimes \Omega^{n-i-1})(a_{n-i-1}))$
is in ${\rm Ker}(\alpha_{n+j-i-1} \otimes \Omega^{n-j})$.

By the Snake Lemma 
${\rm Ker}(\alpha_{n+j-i-1} \otimes \Omega^{n-j}) = (\beta_{n+j-i} \otimes \Omega^{n-j})({\rm Ker}(\alpha_{n+j-i} \otimes \Omega^{n-j})).$
Therefore we have that the element 
$\beta^{i-1} \circ \theta^{n-2} \circ (s \otimes \Omega)(a) + \beta^{i-1} \circ \theta^{n-3} \circ (s \otimes \Omega^2)(a_2) + \dots + \beta^{i-1} \circ \theta^i \circ (s \otimes \Omega^{n-i-1})(a_{n-i-1})$
is in ${\rm Ker}(\alpha_{n+j-i} \otimes \Omega^{n-j}) + {\rm Ker}(\beta_{n+j-i} \otimes \Omega^{n-j})$.

By successive applications of the Snake Lemma and the indutive hypothesis in $j$ about the injectivity of the previous maps we obtain that
$(\alpha_{n+j-i+1} \otimes \Omega^{n-j-1})({\rm Ker}(\beta_{n+j-i-1} \otimes \Omega^{n-j-1})) = {\rm Ker}(\beta_{n+j-i} \otimes \Omega^{n-j})$, ... , 
$(\alpha_n \otimes \Omega^{n-i})(\rm Ker(\beta_n \otimes \Omega^{n-i})) = {\rm Ker}(\beta_{n-1} \otimes \Omega^{n-i+1})$.

Combined these equations give us that $\beta^{i-1} \circ \theta^{i-1} \circ (s \otimes \Omega^{n-i})({\rm Ker}(\beta_n \otimes \Omega^{n-i}))$ 
is in ${\rm Ker}(\beta_{n+j-i} \otimes \Omega^{n-j})$.
Hence there is an element $a_{n-i}$ such that 
$\beta^{i-1} \circ \theta^{n-2} \circ (s \otimes \Omega)(a) + \beta^{i-1} \circ \theta^{n-3} \circ (s \otimes \Omega^2)(a_2) + \dots + \beta^{i-1} \circ \theta^i \circ (s \otimes \Omega^{n-i-1})(a_{n-i-1}) + \beta^{i-1} \circ \theta^{i-1} \circ (s \otimes \Omega^{n-i})(a_{n-i})$
is in ${\rm Ker}(\alpha_{n+j-i} \otimes \Omega^{n-j})$. Therefore we proved the Claim by decreasing induction on $i$.

Note that at the last stage of this induction we obtain the following statement:
$\theta^{n-2} \circ (s \otimes \Omega)(a) + \theta^{n-3} \circ (s \otimes \Omega^2)(a_2) + \dots + \theta^j \circ (s \otimes \Omega^{n-j-1})(a_{n-j-1})$
is in ${\rm Ker}(\alpha_n \otimes \Omega^{n-j}) + {\rm Ker}(\beta_n \otimes \Omega^{n-j})$.

So there is an element $a_{n-j}$ such that 
$$\theta^{n-2} \circ (s \otimes \Omega)(a) + \theta^{n-3} \circ (s \otimes \Omega^2)(a_2) + \dots + \theta^j \circ (s \otimes \Omega^{n-j-1})(a_{n-j-1}) + (s \otimes \Omega^{n-j})(a_{n-j})$$
is in ${\rm Ker}(\alpha_n \otimes \Omega^{n-j})$.

Therefore
$$(\theta \otimes \Omega^{n-j})(\theta^{n-2} \circ (s \otimes \Omega)(a) + \theta^{n-3} \circ (s \otimes \Omega^2)(a_2) + \dots + \theta^j \circ (s \otimes \Omega^{n-j-1})(a_{n-j-1}) + (s \otimes \Omega^{n-j})(a_{n-j}))$$
is in ${\rm Ker}(\beta_n \otimes \Omega^{n-j+1})$.

Hence there is an element $a_{n-j+1}$ such that  
$(\theta \otimes \Omega^{n-j})(\theta^{n-2} \circ (s \otimes \Omega)(a) + \theta^{n-3} \circ (s \otimes \Omega^2)(a_2) + \dots + \theta^j \circ (s \otimes \Omega^{n-j-1})(a_{n-j-1}) + (s \otimes \Omega^{n-j})(a_{n-j})) + (s \otimes \Omega^{n-j+1})(a_{n-j+1}) = 0.$

Let $(\theta \otimes \Omega^{n-j}) \circ \theta^k = \widetilde{\theta}^k$ for $k = j , j+1 , ... , n-3 , n-2$. Also let $\theta \otimes \Omega^{n-j} = \widetilde{\theta}^{j-1}$ and $\widetilde{\theta}^{j-2} = id$. From the previous equation we obtain 
\begin{equation} \label{tilde}
\widetilde{\theta}^{n-2} \circ (s \otimes \Omega)(a) + \widetilde{\theta}^{n-3} \circ (s \otimes \Omega^2)(a_2) + \dots + \widetilde{\theta}^j \circ (s \otimes \Omega^{n-j-1})(a_{n-j-1}) 
\end{equation}
$$
+ \widetilde{\theta}^{j-1}(s \otimes \Omega^{n-j})(a_{n-j}) + \widetilde{\theta}^{j-2}(s \otimes \Omega^{n-j+1})(a_{n-j+1})=0.
$$

Let $a_1 = a$. Let $p$ be the greatest number in the set $\{1,2,...,n-j+1\}$ such that $a_p \neq 0$ (note that $p$ exists because $a_1 = a \neq 0$).

Let $K_1 \subset K_2 \subset \dots \subset K_n =M_n$ be the HN filtration of $M_n$. We have for $q<p$ that 
$\widetilde{\theta}^{n-q-1}(s \otimes \Omega^q)(a_q) \in K_q \otimes \Omega^{n-j+1} \subseteq K_{p-1} \otimes \Omega^{n-j+1}.$ 

However the induced map $\bar{\widetilde{\theta}}^{n-p-1}: K_n/K_{n-1} \to K_p/K_{p-1} \otimes \Omega^{n-j+1}$ is injective. Therefore since $a_p \neq 0$ and $s \otimes \Omega^p$ is injective when composed with the quotient map from $K_n \otimes \Omega^{n-j+1}$ to $K_n \otimes \Omega^{n-j+1}/K_{n-1} \otimes \Omega^{n-j+1}$ we have that $\bar{\widetilde{\theta}}^{n-p-1}(s \otimes \Omega^p)(a_p) \neq 0$ in $K_p \otimes \Omega^{n-j+1}/K_{p-1} \otimes \Omega^{n-j+1}$.

Therefore $\bar{\widetilde{\theta}}^{n-p-1}(s \otimes \Omega^p)(a_p)$ is not in $K_{p-1} \otimes \Omega^{n-j+1}$. This contradicts equation (\ref{tilde}) above. Therefore $\alpha^j: \Omega^{n+1} \to M_j \otimes \Omega^{n-j+1}$ is injective and our induction on $j$ is completed.

\end{proof}

It remains to show that we can also satisfy the condition $\Hom(M_k,\Omega^k)=0$ in the above diagram for all $k$. To do this it is enough to consider sheaves with sections in our HN stratum where the section maps injectively to the smallest HN summand and by zero to the other summands. In this case we can compute all $M_k$ in the diagram on p. 44 in \cite{La} explicitly and check that $\Hom(M_k,\Omega^k)=0$; note that the $M_k$ in this example will have torsion but this is allowed.  

\section{Support of the Drinfeld-Laumon Construction applied to $\W_\rho$}

Consider the stack $F=F_k$ the points of which are given by the data of $(\Omega^{k-1} \to X_k \to X_{k-1} \to ... \to X_1)$ and $L^{(k-i+1)} \to X_i$ for all $i$ such that 
the compositions $L^{(k-i+1)} \to X_i \to X_{i-1}$ are zero and the degree of $L^{(k-i)}$
is $n_{k-i}-2(k-i-1)$ where $n_{k-i}$ is the length of the corresponding row of 
the Young diagram. We consider the direct
image from this stack to the stack $Coh_k^{\prime}$ of $\Omega^{k-1} \to X_k$ of the sheaf which is the pull-back 
(via the map which sends the above point to $X_1/\Omega^{k-1}$) from $Coh_0$ 
(of the sheaf corresponding to the representation of the symmetric group obtained by deleting the
 largest $k-1$ rows of the original Young diagram) 
to an open substack (where this map is well-defined and smooth)  and extended by the 
Goresky-MacPherson extension.

Let $F_0$ be the open substack of $F$ the points of which are the data of $(\Omega^{k-1} \to X_k \to X_{k-1} \to ... \to X_1)$ and $L^{(k-i+1)} \to X_i$ for all $i$ such that the sequences $L^{(k-i+1)} \to X_i \to X_{i-1}$ are short exact.

\begin{lemma} \label{supportlemma} The result of the $k$-th step in the Drinfeld-Laumon construction applied to the sheaf corresponding
to our Young diagram $\mu$ is a direct summand in $\pi_* i_{!*} p^*(S_{\mu^{\prime}})$ in the following diagram,
where $\mu^{\prime}$ is the Young diagram obtained from $\mu$ by deleting the $k-1$ longest rows and $S_{\mu^{\prime}}$
is the corresponding sheaf on $Coh_1^{\prime}$.

\begin{tikzpicture}
  \matrix (m) [matrix of math nodes,row sep=2em,column sep=2em,minimum width=2em, ampersand replacement=\&]
  {
     Coh_1^{\prime} \& F_0 \& F \& Coh_k^{\prime}\\};
  \path[-stealth]
          
    (m-1-2) edge node [above] {$p$} (m-1-1)

     (m-1-3) edge node [above] {$\pi$} (m-1-4);

    \path[right hook->] 

     (m-1-2) edge node [above] {$i$} (m-1-3);

\end{tikzpicture}

\end{lemma}

\begin{proof}

Let the stack $B$ have points given by the data of $L^{(k-i+1)} \to X_i$ and $X_i \to X_{i-1}$ for all $i$ 
such that the compositions $L^{(k-i+1)} \to X_i \to X_{i-1}$ are all zero. Let the open substack $B_0$ of $B$
be given by the data above such that $L^{(k-i+1)} \to X_i \to X_{i-1}$ are all short exact sequences.
Let $U_0$ be the vector bundle over $B_0$ with fiber $\Hom(\Omega^k, X_1)$. We have the following diagram

\begin{tikzpicture}
  \matrix (m) [matrix of math nodes,row sep=2em,column sep=2em,minimum width=2em, ampersand replacement=\&]
  {
     Coh_1^{\prime} \& U_0 \& F_0 \& F \\
     \ \& \ \& B_0 \& B \\  };
  \path[-stealth]
          
    (m-1-2) edge node [above] {$p^0$} (m-1-1)
                 edge (m-2-3)

     (m-1-3) edge (m-1-2)
                  edge (m-2-3)

     (m-1-4) edge (m-2-4);

    \path[right hook->] 

     (m-1-3) edge (m-1-4)

     (m-2-3) edge (m-2-4);

\end{tikzpicture}

Assuming by induction that the result of the $k$-th step of the Drinfeld-Laumon construction is 
$\pi_* i_{!*} p^*(S_{\mu^{\prime}})$, we will prove an analogous statement for the $(k+1)$-th step;
note that the base of induction is furnished by the Main Theorem of Chapter 1.
For this we consider the following dual diagram:

\begin{tikzpicture}
  \matrix (m) [matrix of math nodes,row sep=2em,column sep=2em,minimum width=2em, ampersand replacement=\&]
  {
     \ \& \overline{U^{*,new}_0} \& \overline{U_0^*} \& \ \& \ \\
     Coh_1^{\prime} \& U^{*,new}_0 \& U_0^* \& F_0^* \& F^* \\
     \ \& \ \& \ \&  B_0 \& B \\  };
  \path[-stealth]
          
    (m-2-2) edge node [above] {$p^{\prime}$} (m-2-1)
                 edge node [above] {$f$} (m-2-3)

     (m-1-2) edge (m-1-3)

     (m-2-4) edge (m-3-4)

     (m-2-2) edge (m-3-4)

     (m-2-3) edge (m-3-4)

     (m-2-5) edge (m-3-5);

    \path[right hook->] 

     (m-1-3) edge (m-2-5)

     (m-2-2) edge (m-1-2)

     (m-2-3) edge (m-1-3)

     (m-2-3) edge (m-2-4)

     (m-2-4) edge (m-2-5)

     (m-3-4) edge (m-3-5);

\end{tikzpicture}

In this diagram $F^*$ is the dual bundle of the bundle $F$ over $B$; $F_0^*$ and $U_0^*$
are the dual bundles of the bundles $F_0$ and $U_0$ over $B_0$. Moreover, $U^{*,new}_0$
is obtained by adding to the data of $U^*_0$ of a map $L^{(k)} \to X_2$. Likewise, $\overline{U^{*,new}_0}$
is obtained by adding the data of a map $L^{(k)} \to X_2$ to the data of $\overline{U_0^*}$.

Note that when we perform the Fourier-Deligne transform in the bundle $U_0$ over $B_0$
of the sheaf $p_0^* (S_{\mu})$ we obtain the sheaf $f_* (p^{\prime})^* (S_{\mu_1})$ where
$\mu_1$ is obtained by deleting the longest row of $\mu$.
This follows by using the first step of the argument, the Main Theorem of Chapter 2, in families over $B_0$.

Let $\pi^{\vee}: F^* \to \  ^0 Coh_{k+1}^{\prime}$ be the map which sends the point 
given by $L^{(i)} \to X_{k+1-i}$ and $\Omega^k \hookrightarrow X_{k+1}^{\prime} \to X_k$
to $\Omega^k \hookrightarrow X_{k+1}^{\prime}$. Note that performing the Fourier transform
over the base $B$ can be replaced by performing it over the base $B_0$ followed by an intermediate extension. 
We obtain that the Fourier transform of the result at the $k$-th step above is the sheaf 
$\pi^{\prime}_*  i^{\prime}_{!*} (p^{\prime})^* (S_{\mu_1})$ 
for the following maps in the above diagram
(followed by the composition of the map $\overline{U^{*,new}_0} \hookrightarrow F^*$ of the 
above diagram and the map $ \pi^{\vee}: F^* \to Coh_{k+1}^{\prime,0} $):

\begin{tikzpicture}
  \matrix (m) [matrix of math nodes,row sep=2em,column sep=2em,minimum width=2em, ampersand replacement=\&]
  {
     Coh_1^{\prime} \& U^{*,new}_0 \& \overline{U^{*,new}_0} \& Coh_{k+1}^{\prime, 0}\\};
  \path[-stealth]
          
    (m-1-2) edge node [above] {$p^{\prime}$} (m-1-1)

     (m-1-3) edge node [above] {$\pi^{\prime}$} (m-1-4);

    \path[right hook->] 

     (m-1-2) edge node [above] {$i^{\prime}$} (m-1-3);

\end{tikzpicture}

After an intermediate extension these maps will be replaced by the maps in the diagram

\begin{tikzpicture}
  \matrix (m) [matrix of math nodes,row sep=2em,column sep=2em,minimum width=2em, ampersand replacement=\&]
  {
     Coh_1^{\prime} \& F_{k+1,0} \& F_{k+1} \& Coh_{k+1}^{\prime}\\};
  \path[-stealth]
          
    (m-1-2) edge node [above] {$p_{k+1}$} (m-1-1)

     (m-1-3) edge node [above] {$\pi_{k+1}$} (m-1-4);

    \path[right hook->] 

     (m-1-2) edge node [above] {$i$} (m-1-3);

\end{tikzpicture}

hence we proved the desired statement by induction on $k$.

\end{proof}

From the lemma above the following is immediate:

\begin{proposition} \label{remark} For any Young diagram with rows of lengths $n_1 > n_2 + (2g-2) > ... > n_k + (2g-2) > ...$
where $n_k > n_{k+1} + n_{k+2} + ... + (2g-2)$,
the $k$-th step of the Drinfeld-Laumon construction applied to the sheaf corresponding to this Young diagram
is supported on the closure of the locus of bundles with section such that the bundle has a filtration with
subquotients of rank 1 and degrees $n_i +(i-1)(2g-2)$ and $k(2g-2) +n_{k+1}+n_{k+2}+...$.
\end{proposition}

\section{Induction on the Young diagram}

\begin{lemma} \label{indlemma} Let $D$ be a Young diagram such that the length of its longest row is much greater than the sum of the lengths of the rest of the rows. Let $D^{\prime}$ be the Young diagram obtained from $D$ by deleting the longest row of $D$. Suppose that we know that the result of $k$ steps of the Drinfeld-Laumon construction applied to the sheaf $\W_{D^{\prime}}$ on $\Coh_1^{\prime}$ is the sheaf $\W_{S^{\prime},D_{\rm red}}$ on the corresponding Harder-Narasimhan stratum in $\Coh_k^{\prime}$ as in the statement of the Main Theorem. Then the result of $k+1$ steps of the Drinfeld-Laumon construction applied to $\W_D$ is as in the statement of the Main Theorem, that is equal to the sheaf $\W_{S,D_{\rm red}}$ where $S$ has the same summands as $S^{\prime}$ plus an additional summand coming from the longest row of $D$.
\end{lemma}

\begin{proof}

First we calculate the result on the set of $L \oplus X$ where $L$ has rank one and comes from the longest row of $D$ for $X$ of rank $1,2,...,k$.
We show by induction by using the Fourier transform trick that after each Fourier transform we can replace push-forward by pull-back, and we show 
that pull-back is preserved after an intermediate extension. 

The base of induction follows from the Main Theorem of \cite{Do}. Namely from the above result we obtain that on the set of bundles of the form 
$L \oplus X$ with an injective section, which is also injective when composed with the projection onto $X$, the sheaf is pulled back from the set of $X$ with an injective section. Since this sheaf coincides with the restriction to this open set of the sheaf on the set of bundles $L \oplus X$ with a section which is the pullback from the set of $X$ with a section of the corresponding sheaf $\W$ there, we obtain the base of induction by the uniqueness of the intermediate extension.

Assume that we have proved the result at the $k$-th step of the Drinfeld-Laumon construction. After the next Fourier transform, by the Fourier transform trick, the pullback from the bundle with fiber ${\rm Hom}(\Omega^{k-2}, X)$ to the bundle with fiber ${\rm Hom}(\Omega^{k-1}, L \oplus X)$ will turn into the pushforward from the bundle with fiber ${\rm Ext}^1(X, \Omega^{k-1})$ to the bundle with fiber ${\rm Ext}^1(L \oplus X, \Omega^k)$. Hence on the 
open subset of the stack of bundles with a section which consists of $\Omega^k \hookrightarrow L \oplus Y$, such that $L$ is a summand in the quotient (note that the above map along which we push forward is one-to-one on this set), the sheaf is the restriction of the sheaf which is the pullback from the set of $\Omega^k \to Y$ to the set of $\Omega^k \to L \oplus Y$. Hence we obtain the step of induction by the uniqueness of the intermediate extension, using the fact that after the $(k+1)$-th step of the DL construction the resulting sheaf is supported on the closure of the locus of sheaves of the form $L \oplus Y$ which follows from Lemma \ref{supportlemma}.

(Note that the set where the cokernel has a summand $L$ is open because its complement consists of bundles with a section where the cokernel has a summand which is a line bundle of degree which is higher than the degree of $L$ so $L$ must land there so the quotient of $Y$ by the section has torsion. But the set of injective sections of $Y$ in which cokernel has torsion is closed).

Next we apply the result about the support of the the Drinfeld-Laumon construction to sheaves $\W_D$ and we obtain that it must be supported on the closure of a Harder-Narasimhan stratum but by the above we already know it on an open substack of the stratum, so we are done.

\end{proof}




Now we use Lemma \ref{indlemma} to do the induction on the Young diagram. We will do a double induction: increasing on the number of blocks which are not in the longest row and decreasing on the number of blocks in the longest row. The base of induction is provided by the Main Theorem of \cite{Do}. Now to do the step of induction we assume that the Main Theorem of this chapter is true for all Young diagrams with either at most $m$ blocks that are not in the longest row or exactly $m$ blocks not in the longest row and at least $N+1$ blocks in the longest row. (Note that for Young diagrams with $m$ blocks not in the longest row and $>>m$ blocks in the longest row the statement follows from the statement for smaller $m$ by Lemma \ref{indlemma} since the diagram obtained from $D$ by throwing away the longest row has $<m$ blocks not in its first row, so for it the statement is known by the inductive assumption). 

Let us take a diagram $D$ with exactly $m$ blocks not in the longest row and $N$ blocks in the longest row and prove the Main Theorem for it. First add some blocks to the longest row of $D$ so that the length of the first row of the resulting diagram $D_1$ is $N+1>>m$. Let us prove the statement for the diagram $D_2$ which is obtained from $D_1$ by deleting one square in the longest row. We apply the Hecke functor to $\W_{D_1}$ and using the result of Section \ref{hecke_1} by Pieri's rule we obtain the direct sum of sheaves $\W_{D^{\prime}}$ for diagrams $D^{\prime}$ that are obtained from $D$ by deleting one square. Note that $D_2$ is one of such diagrams and all diagrams $D^{\prime}$ except $D_2$ have fewer than $m$ squares not in the longest row, so the inductive hypothesis applies to them. Hence we know the result of the Drinfeld-Laumon construction for $D_1$ and for all $D^{\prime}$ except $D_2$.

By the computation in Section \ref{hecke_k} we know that the Hecke functor applied to the result of the DL construction applied to $\W_D$ has a summand which is the conjectured answer for $D_2$. By the result of Section \ref{hecke_commute} this summand is either the result of the DL construction applied to one of the $\W_{D^{\prime}}$ or it is degenerate. But it cannot be degenerate by Section \ref{nondeg} and it cannot be the result of the DL construction applied to $D^{\prime}$ other than $D_2$ because they are supported on smaller HN strata (because the diagrams have greater longest rows that $D_2$). Hence it must be the DL construction applied to $\W_{D_2}$ and we are done with the induction step.

\end{document}